\documentclass[a4paper,12pt]{amsart}

\usepackage[english]{babel}
\usepackage{amsthm}
\usepackage{amssymb}
\usepackage{hyperref}

\newcommand{\Set}[1]{\left\{\, #1 \,\right\}}

\newcommand{\Span}[1]{\langle\, #1 \,\rangle}

\newcommand{\Order}[1]{\lvert #1 \rvert}
\newcommand{\Index}[2]{\lvert #1 : #2 \rvert}
\newcommand{\Size}[1]{\lvert #1 \rvert}

\DeclareMathOperator{\End}{End}
\DeclareMathOperator{\Frat}{Frat}
\DeclareMathOperator{\Hom}{Hom}

\newcommand{\fpf}[0]{fixed-point-free}

\renewcommand{\phi}[0]{\varphi}
\renewcommand{\theta}[0]{\vartheta}
\renewcommand{\epsilon}[0]{\varepsilon}

\newtheorem{dummy}{Dummy}
\numberwithin{dummy}{section}
\numberwithin{equation}{section}

\newtheorem{theorem}[dummy]{Theorem}
\newtheorem{lemma}[dummy]{Lemma}
\newtheorem{prop}[dummy]{Proposition}

\theoremstyle{definition}

\newtheorem{definition}[dummy]{Definition}

\newcommand{\FMEO}[1]{}

\begin{document}

\bibliographystyle{amsalpha}

\date{14 December 2012, 11:29 CET --- Version 5.01%
}

\title[Quasi-inverse endomorphisms]
{Quasi-inverse endomorphisms}

\author{A.~Caranti}

\address[A.~Caranti]%
 {Dipartimento di Matematica\\
  Universit\`a degli Studi di Trento\\
  via Sommarive 14\\
  I-38123 Trento\\
  Italy} 

\email{andrea.caranti@unitn.it} 

\urladdr{http://science.unitn.it/$\sim$caranti/}

\begin{abstract}
Greither   and    Pareigis have established a
connection between Hopf Galois structures on a Galois extension $L/K$
with Galois group $G$, and the regular subgroups of the group of
permutations on $G$, which are normalized by $G$.
Byott has rephrased this connection in terms of certain
equivalence classes of
injective morphisms of $G$ into the holomorphs of the groups $N$ with the same
cardinality of $G$. 

Childs and Corradino have used this theory to
construct such Hopf Galois structures, starting from \fpf\ endomorphisms
of $G$ that have abelian images.
In this paper  we show that a \fpf\ endomorphism  has an abelian image
if and only if there is  another endomorphism that is its inverse with
respect to the  circle operation in the near-ring of  maps on $G$, and
give  a   fairly  explicit   recipe  for  constructing   all  such
endomorphisms.
\end{abstract}


%% \thanks{Credit where is due}
%% {First author partially supported by MIUR-Italy via PRIN 2006014340-002
%%  ``Lie algebras and rings. Groups. Cryptography''.  Second  author partially  supported by
%%   MIUR-Italy via PRIN 2007 ``Group theory and applications''.}
  

\maketitle

\thispagestyle{empty}

\section{Introduction}

Our starting point is  the paper~\cite{Childs}, in which Childs begins
by  reviewing  the  theory  of  Greither  and  Pareigis~\cite{GP}. This  theory establishes a bijection between
Hopf  Galois structures  on a  Galois extension  $L/K$ of  fields with
Galois  group  $G$,  and  the   regular  subgroups  of  the  group  of
permutations  on $G$,  which are  normalized  by $G$.
Byott~\cite{By96a, By96b} has shown that the problem  of
determining these  regular subgroups can  be  translated into that
of finding certain equivalence classes  of injective morphisms of $G$ into
the  holomorphs  of  the  groups  $N$  with  the  same  cardinality  of
$G$. Childs and Corradino~\cite{ChiCor} and Childs~\cite{Chi07} have showed
that abelian \fpf\ endomorphisms of  $G$ yield directly, via the above
theory, such  Hopf Galois  structures. (Here we  say with~\cite{Childs}
that an  endomorphism is  \emph{abelian} if its  image is  abelian.) We
refer to~\cite{Childs}~and \cite{Chi00} for the details.

Childs studies  in particular those Hopf Galois  structures that arise
from  the abelian  \fpf\  endomorphism  $\phi$ of  $G$  that admit  an
inverse  endomorphism   $\psi$  with  respect   to  the  circle
operation $\phi  \circ \psi = \psi  - \phi \psi +  \phi$. (Childs
calls $\psi$ the \emph{inverse} of $\phi$. Since he only 
considers \fpf\ endomorphisms that  are not automorphisms, the term
is unambiguous in  his context. In this paper we will
have to deal also with \fpf\ automorphisms, and 
with inverses with respect to map composition. We have thus preferred
to use the classical term  \emph{quasi-inverse} for  $\psi$ \cite[Chap.~1,
  Section.~5]{SoR}.)

In this paper we study the  \fpf\ endomorphisms of a finite group that
have a quasi-inverse. 
We will show that for a \fpf\
endomorphism of a finite group the properties of being abelian, and
that of
having a quasi-inverse, are equivalent (Theorem~\ref{thm:recab},
which extends \cite[Remark 
  10]{Childs}, and shows that the condition of having a quasi-inverse is not
restrictive in the context of~\cite{Childs}). 
In  Sections~\ref{sec:Fitting}~and \ref{sec:piecingtogether} we are
able to  give reasonably 
explicit  recipes  for  constructing  the  groups that  have  such  an
endomorphism, and for determining all of their endomorphisms with this
property.    Our  main  tool  is  a version  for  groups  of
Fitting's  Lemma  for  modules (Section~\ref{sec:Fitting}). We are
then led to
study
\fpf\ endomorphisms of finite abelian groups (Section~\ref{sec:ab}),
abelian nilpotent endomorphisms (Section~\ref{sec:nilp}),
and how to put the two together (Section~\ref{sec:piecingtogether}).
In Section~\ref{sec:ex} we  will
be  reviewing  the  examples 
of~\cite{Childs} from our point of view, and provide some more.

\section{Preliminaries}

Let $G$ be a group, multiplicatively written. 
In our context, it is convenient to write maps on $G$ as
exponents. Given two maps $\phi, \psi$ on $G$, 
we define as usual their pointwise sum $\phi + \psi$ by $x^{\phi +
  \psi} = x^{\phi} x^{\psi}$, for $x \in G$ and their product  $\phi
\psi$ as their
composition, $x^{\phi \psi} = (x^{\phi})^{\psi}$. It is well known
that these operations define a structure of a near-ring on the set
$N(G)$ of
maps on $G$. (See~\cite{Pilz}.) The sum  is not commutative in
general, and only the 
distributive relation $\phi (\psi + \theta) = 
\phi \psi + \phi \theta$ holds generally, while if $\phi$ is an
endomorphism of $G$, we also have $(\psi + \theta) \phi =
\psi \phi + \theta \phi$. Note also that for a map $\phi$ on $G$, $-
\phi$ is defined by $\phi + (-\phi) = 0$ (where $x^{0} = 1$ for all
$x \in G$), so that 
$x^{-\phi} = (x^{\phi})^{-1}$ for $x \in G$. Thus $- (\phi +
\psi) = - \psi - \phi$ for maps $\phi, \psi$ on $G$.

If $\phi, \psi$ are endomorphisms of $G$, their sum $\phi + \psi$ is
not in general an endomorphism of $G$. 
%% (For example, if $1$ is the
%% identity map on $G$, then $1 + 1$ is the squaring operation, and this
%% is an endomorphism of $G$ if and only if $G$ is abelian.) 
We have
\begin{prop}\label{prop:commute}
  Let $G$ be a group, $\phi, \psi \in \End(G)$. The following are
  equivalent
  \begin{enumerate}
  \item $\phi + \psi \in \End(G)$, and
  \item $[G^{\phi}, G^{\psi}] = 1$, that is, the images of $\phi$ and
    $\psi$ commute elementwise.
  \end{enumerate}
  Moreover, these conditions imply $\phi + \psi = \psi + \phi$.
\end{prop}

\begin{proof}
$
  (a b)^{\phi + \psi} 
  = 
  (a b)^{\phi} (a b)^{\psi}
  =
  a^{\phi} b^{\phi} a^{\psi} b^{\psi}
$
equals
$
  a^{\phi + \psi} b^{\phi + \psi}
  =
  a^{\phi} a^{\psi} b^{\phi} b^{\psi}
$
for all $a, b \in G$ if and only if $b^{\phi} a^{\psi} =
a^{\psi} b^{\phi}$ for all $a, b \in G$. Setting $b = a$ in the last
identity we obtain $\phi + \psi = \psi + \phi$.
\end{proof}

\begin{definition}
  Let $\phi$ be an endomorphism of the group $G$. A \emph{fixed point}
  of $\phi$ is an element $x \in G$ such that $x^{\phi} = x$. We say
  that $\phi$ is \emph{\fpf} if its only fixed point is $1$.
\end{definition}

The following is well-known \cite[Lemma~1.1., Chap.~10]{Gore}.
\begin{prop}\label{prpo:fpfbi}
  Let $G$ be a finite group, $\phi \in \End(G)$.
  The following are equivalent:
  \begin{enumerate}
  \item the map $1 - \phi$ is injective, and thus a bijection on $G$,
    and
  \item $\phi$ is \fpf.
  \end{enumerate}
\end{prop}

\begin{proof}
  $x^{1 - \phi} = x x^{-\phi} = y y^{-\phi} = y^{1 - \psi}$ if and
  only if
  $y^{-1} x = (y^{-1} x)^{\phi}$, that is, $y^{-1} x$ is a fixed point
  of $\phi$.
\end{proof}

\section{Quasi-inverse endomorphisms}
\label{sec:rec}

Childs studies in~\cite{Childs} the \fpf\ endomorphisms $\phi$ such that $1
- \phi$ has an inverse of the same form $1 - \psi$, for some
endomorphism $\psi$ of $G$. 
\begin{definition}\label{def:inverse}
  Let $G$ be  a finite group. A \fpf\  endomorphisms $\phi$ of
  $G$ is said to be \emph{quasi-invertible} if there is an
  endomorphism $\psi$ of $G$ such that the map $1  - \psi$ is the
  inverse of the map $1 - \phi$ with respect to map
  composition. $\psi$ is said to be the \emph{quasi-inverse} of
  $\phi$. 
\end{definition}
Clearly $\psi$ is also
\fpf\ by Proposition~\ref{prpo:fpfbi}, and $\psi$ is uniquely
determined by $\phi$. 
A familiar argument shows that if $\psi$ is the quasi-inverse of
$\phi$, then 
$1 = (1 - \phi)(1 - \psi) = 1 - \phi + \phi \psi - \psi$,
so that $-\phi + \phi \psi - \psi = 0$, or
\begin{equation}\label{eq:main}
  \phi + \psi = \phi \psi.
\end{equation}
In other words,  $\psi$ is the inverse of $\phi$ with
respect to the 
circle operation  $\phi \circ \psi = \psi  - \phi \psi +  \phi$ on the
near-ring $N(G)$.   Thus $\phi$  is \emph{right quasi-regular}  in the
classical radical theory  of Jacobson \cite[Chap.~1, Section.~5]{SoR},
the catch here being that  we require the  (right) \emph{quasi-inverse}
$\psi$ of $\phi$ to be another endomorphism of $G$.

Since $\phi + \psi = \phi \psi \in \End(G)$, we have by
Proposition~\ref{prop:commute} $\phi + \psi = \psi + \phi$, and thus
$\phi \psi = \psi \phi$. Also, $[G^{\phi}, G^{\psi}] = 1$. Now we
can rewrite~\eqref{eq:main} as
\begin{equation*}
  \phi = - \psi + \phi \psi = (-1 + \phi) \psi,
\end{equation*}
as $\psi \in \End(G)$. This shows that $G^{\phi} \subseteq G^{\psi}$
and then by symmetry (or because $-1 + \phi$ is a bijection) $G^{\phi}
=  G^{\psi}$. Also, if $x \in \ker(\phi)$, then $1 = x^{\phi} = x^{(-1 +
\phi) \psi} = (x^{-1})^{\psi}$, so that $x \in \ker(\psi)$, and thus
$\ker(\phi) = \ker(\psi)$. We have obtained 
\begin{prop}\label{prop:abimage}
  Let $G$ be a finite group, $\phi, \psi$ two \fpf\ endomorphisms of
  $G$ that are one the quasi-inverse of the other. Then 
  \begin{enumerate}
  \item $\phi + \psi = \psi + \phi$ and $\phi \psi = \psi \phi$,
  \item   $G^{\phi} = G^{\psi}$ is an abelian
    subgroup of $G$, and 
  \item $\ker(\phi) = \ker(\psi)$.
  \end{enumerate}
\end{prop}
We take, as in \cite[Definition~1]{Childs}, the following
\begin{definition}
  An endomorphism $\phi$ of the group $G$ is said to be \emph{abelian}
  is its image $G^{\phi}$ is abelian, or equivalently $G' \le \ker(\phi)$.
\end{definition}
Thus all quasi-invertible \fpf\ endomorphisms $\phi$ of the
finite group $G$  are abelian.
We have in fact 
\begin{theorem}\label{thm:recab}
  Let $G$ be a finite group, $\phi$ a \fpf\ endomorphism of $G$. The
  following are equivalent:
  \begin{enumerate}
  \item $\phi$ is quasi-invertible, and
  \item $\phi$ is abelian.
  \end{enumerate}
\end{theorem}
We defer the proof to Section~\ref{sec:piecingtogether}.

%% \begin{proof}
%%   It remains to show that if $\phi$ is abelian, then $\phi$ has a
%%   reciprocal. As shown in~\cite[Remark~10]{Childs} (and reprised below
%%   in Section~\ref{sec:ab}), $\phi$ restricted to the abelian group
%%   $G^{\phi}$ has a 
%%   reciprocal $\psi = - (1 - \phi)^{-1} + 1$. Clearly such a $\psi$ is
%%   defined on the whole of $G$, and $(1 - \phi)^{-1} = 1 - \psi$, so
%%   that $\phi$ and $\psi$ are reciprocal.
%% \end{proof}

\section{Fitting's Lemma}
\label{sec:Fitting}

The following is a standard result in the theory of modules
\cite[p.~113]{BAII}. 
\begin{theorem}[Fitting's Lemma]
  Let $A$ be a ring, and $M$ an $A$-module that is both artinian and
  noetherian. 

  Let $\phi$ be an endomorphism of $M$. 
  Then there is a a natural
  number $n$ such that $\ker(\phi^{n}) = \ker(\phi^{n+i})$ and $M^{\phi^{n}}
  = M^{\phi^{n+i}}$ for each $i \ge 0$. We have
  \begin{equation*}
    M = \ker(\phi^{n}) \oplus G^{\phi^{n}}.
  \end{equation*}
  Moreover, the restriction of $\phi$ to $\ker(\phi^{n})$ is
  nilpotent, and the restriction of $\phi$ to $G^{\phi^{n}}$ is an
  automorphism.
\end{theorem}

The following version for groups is folklore.
\begin{theorem}[Fitting's Lemma for groups]
  \label{thm:Fitting}
  Let $G$ be a group  which satisfies the ascending chain condition on
  normal subgroups, and the descending chain condition on subgroups.

  Let $\phi$ be an endomorphism of $G$.   Then there is
  a natural 
  number $n$ such that $\ker(\phi^{n}) = \ker(\phi^{n+i})$ and $G^{\phi^{n}}
  = G^{\phi^{n+i}}$ for each $i \ge 0$.
  We have:
  \begin{enumerate}
  \item $G$ is
    the semidirect product of the normal subgroup $K = \ker(\phi^{n})$ by
    the subgroup $H = G^{\phi^{n}}$,
  \item the restriction of  $\phi$ to $K$
    is nilpotent, and 
  \item the restriction of  $\phi$ to  $H$  is an
    automorphism.
  \end{enumerate}
\end{theorem}
We give a proof that is basically the one recorded in~\cite{CarComm}
for finite groups. 
\begin{proof}
  An integer $n$ as in the statement exists because of the chain conditions.
  Set  $H  =  G^{\phi^{n}}$, $K  =
  \ker(\phi^{n})$.  We claim that $G$ is the semidirect product of $K$ by
  $H$.  In fact,  $K = \ker(\phi^{n})$ is normal in $G$.   For each $x \in
  G$, we have  $x^{\phi^{n}} \in G^{\phi^{n}} = G^{\phi^{2 n}}$.   Thus there is $y
  \in G$ such that $x^{\phi^{n}} = y^{\phi^{2 n}}$.  Then
  \begin{equation*}
    (y^{-\phi^{n}} x)^{\phi^{n}} = y^{-\phi^{2 n}} x^{\phi^{n}} = 1
  \end{equation*}
  that is, $y^{-\phi^{n}} x \in K$, so that $x = y^{\phi^{n}}
  (y^{-\phi^{n}} x)  \in H K$, and $G = H K$. 
  Now $\phi$ is surjective on $H = G^{\phi^{n}}
  = G^{\phi^{n+1}}$. If $x \in H \cap K$,
  then there is $y \in G$ such that $x = y^{\phi^{n}}$, and
  $y^{\phi^{2 n}} = x^{\phi^{n}} = 1$, so that $y 
  \in \ker(\phi^{2 n}) = \ker(\phi^{n})$,
  and $x = 1$. It follows that $H \cap K = 1$, and since $\ker(\phi)
  \le K$, also that $\phi$ is
  injective on $H$.
\end{proof}

%% Keeping in mind Proposition~\ref{prop:abimage}, we obtain the immediate
%% consequence 
%% \begin{theorem}\label{thm:main}
%%   Let $G$ be a finite group, $\phi$ a \fpf\ endomorphism of $G$ that
%%   has a reciprocal.
  
%%   Then $G$ is the semidirect product of a normal subgroup $K$ by an
%%   abelian subgroup $H$ such that 
%%   \begin{enumerate}
%%   \item both subgroups $H$ and $K$ are $\phi$-invariant,
%%   \item $\phi$ is nilpotent on $K$, and
%%   \item $\phi$ is an automorphism on $H$.
%%   \end{enumerate}
%% \end{theorem}

Proposition~\ref{prop:abimage} yields readily
\begin{lemma}\label{lemma:Fitting}
  If the \fpf\ endomorphism $\phi$ of the finite group $G$ has a
  quasi-inverse $\psi$, 
  then  $\psi$ acts on $K = \ker(\phi^{n})$ and $H = G^{\phi^{n}}$.
\end{lemma}

Theorem~\ref{thm:Fitting}~and Lemma \ref{lemma:Fitting} now yield two
recipes,
\begin{enumerate}
\item one for constructing all finite groups that
  have a quasi-invertible \fpf\ endomorphisms,
\item the other for constructing all such endomorphisms for a given finite
  group.
\end{enumerate}

The first recipe is the following.
\begin{enumerate}
\item Take an arbitrary finite group $K$ and a finite abelian group
  $H$.
\item Take a \fpf\ automorphism $\theta$ of $H$ (Section~\ref{sec:ab}).
\item Take a nilpotent endomorphism  $\eta$ of $K$ which is abelian,
  that is, such that $K' \le
  \ker(\eta)$  (Section~\ref{sec:nilp}).  
\item Construct a semidirect product $G$  of $K$ by $H$ such that $[K,
  H] \le \ker(\eta)$ and $[H^{\theta},  K^{\eta}] = 1$, and define the
  quasi-invertible   \fpf\  endomorphism   $\phi$  of   $G$   via  its
  restrictions    $\eta$    on     $K$    and    $\theta$    on    $H$
  (Theorem~\ref{thm:glue} of Section~\ref{sec:piecingtogether}).
\end{enumerate}

If we are already given a finite group $G$, this reads as follows.
\begin{enumerate}
\item Write $G$ as the semidirect
  product of a normal subgroup $K$ by an abelian subgroup $H$.
\item Take a \fpf\ automorphism $\theta$ of $H$.
\item Take a  nilpotent endomorphism $\eta$ of $K$ such  that $G' = K'
  [K, H] \le \ker(\eta)$ and $[H^{\theta}, K^{\eta}] = 1$.
\item Define the quasi-invertible \fpf\ endomorphism $\phi$ of $G$ via
  its restrictions $\eta$ on $K$ and $\theta$ on $H$.
\end{enumerate}
%% Note that $G' = K' [K, H]$, so the two conditions $K' \le \ker(\eta)$
%% and $[K, H] \le \ker(\eta)$ of the first recipe are the same as the
%% single condition $G' \le \ker(\eta)$ of the second one.

%% We obtain a \fpf\ endomorphism of $G$ even if we choose
%% $\theta$ to be a \fpf\ \emph{endomorphism} of $H$.

\section{The abelian case}
\label{sec:ab}

If $\phi$ is a \fpf\ endomorphism of the abelian group $G$, then
$\phi$ clearly is quasi-invertible, as noted
in~\cite[Remark~10]{Childs}. In fact, $1 - \phi$ is an
automorphism of $G$ here, so its inverse (with respect to map
composition) is also an automorphism of 
$G$, and $\psi = - (1 - \phi)^{-1} + 1 \in \End(G)$ is the
quasi-inverse of $\phi$

Since the Sylow $p$-subgroups of $G$ are fully invariant, we need only
consider the case of finite  abelian $p$-groups. We give a description
of the \fpf\ endomorphisms of a finite abelian $p$-group $G$, based on
the approach of~\cite{autoab} to the automorphisms of a finite abelian
$p$-group.

So let $G$ be a finite abelian $p$-group. Write $G$
as the direct product of homocyclic components
\begin{equation}\label{eq:dirsum}
  G = H_{1} \times H_{2} \times \dots \ H_{n},
\end{equation}
where each $H_{i}$ is homocyclic, of exponent $p^{e_{i}}$, with
\begin{equation*}
  0 < e_{1} < e_{2} < \dots < e_{n}.
\end{equation*}
(Clearly this decomposition is not unique in general.)
If $u \in G$, we will write $u_{i}$ for the $i$-th component of $u$
with respect to the decomposition~\ref{eq:dirsum},
that is, each $u_{i} \in H_{i}$ and  $u = u_{1} \cdot u_{2} \cdot \ldots
\cdot u_{n}$. 

Let $\theta_{i} :  H_{i} \to G$ and $\pi_{i}
: G \to  H_{i}$ be the injections and projections  with respect to the
decomposition~\eqref{eq:dirsum}. If $\alpha \in \End(G)$, write
\begin{equation*}
  \alpha_{ij} = \theta_{i} \circ \alpha \circ \pi_{j}
  :
  H_{i} \to H_{j}
\end{equation*}
for the $(i, j)$-th component of $\alpha$ and
\begin{equation*}
  \beta_{i} = \alpha_{ii} \mid_{\Omega_{1}(H_{i})}
\end{equation*}
for the restriction of $\alpha_{ii}$ to $\Omega_{1}(H_{i}) = \Set{ x
  \in H_{i} : x^{p} = 1}$. Then 
we have
\begin{equation}\label{eq:splitend}
  u^{\alpha}
  =
  \prod_{i, j=1}^{n} u_{i}^{\alpha_{ij}}
  =
  \prod_{j=1}^{n} \left( \prod_{i=1}^{n} u_{i}^{\alpha_{ij}} \right),
\end{equation}
with $\prod_{i=1}^{n} u_{i}^{\alpha_{ij}} \in H_{j}$.

Our characterization of \fpf\ endomorphisms of $G$ is the following:
\begin{prop}\label{prop:abfpf}
  In the above notation, the following are equivalent:
  \begin{enumerate}
  \item   $\alpha$ is a \fpf\ endomorphism of $G$, and
  \item each $\beta_{i}$ is a \fpf\ endomorphism of
  $\Omega_{1}(H_{i})$. 
  \end{enumerate}
\end{prop}

Note that we can read off the action of $\beta_{i}$ also on $H_{i} /
H_{i}^{p}$.

\begin{proof}
  Clearly an endomorphism of $G$ has a nontrivial fixed point if and
  only if it has a fixed point of order $p$.
  
  Suppose first that each $\beta_{i}$ is \fpf\ on
  $\Omega_{1}(H_{i})$. Let $u \in \Omega_{1}(G)$ be a 
  fixed point of $\alpha$. If $i > j$, then $p^{e_{i}} = \exp(H_{i}) >
  \exp(H_{j}) = p^{e_{j}}$, so that
  $\Omega_{1}(H_{i}) \le \ker(\alpha_{ij})$. Therefore we have,
  according to~\eqref{eq:splitend},
  \begin{equation}\label{eq:components}
    \left\{
    \begin{matrix}
      u_{1}^{\beta_{1}} && && && &=& u_{1}\\
      u_{1}^{\alpha_{12}} &\cdot& u_{2}^{\beta_{2}} && && &=& u_{2}\\
      \dots\\
      u_{1}^{\alpha_{1n}} &\cdot& u_{2}^{\alpha_{2n}} &\cdot& \dots &\cdot& u_{n}^{\beta_{n}} &=& u_{n}\\
    \end{matrix}
    \right.
  \end{equation}
  so that we have $u_{1} = 0$ from the first identity, and then $u_{2}
  = 0$ from the second one, and so on, so that $u = 0$.

  Conversely, suppose some of the $\beta_{i}$ have nontrivial fixed
  points, and let $k$ be the largest index $i$ for which this happens,
  so that $\beta_{k+1}, \dots , \beta_{n}$ are \fpf. We want to
  construct a fixed point $u$ of $\alpha$ of order $p$. We build $u$
  as $u = u_{k} \cdot u_{k+1} \cdot \ldots \cdot u_{n}$, where the $u_{i} \in
  \Omega_{1}(H_{i})$ are determined as follows. Here~\eqref{eq:components} reads
 \begin{equation}\label{eq:components2}
    \left\{
    \begin{matrix}
      u_{k}^{\beta_{k}} && && && &=& u_{k}\\
      u_{k}^{\alpha_{k,k+1}} &\cdot& u_{k+1}^{\beta_{k+1}} && && &=& u_{k+1}\\
      \dots\\
      u_{k}^{\alpha_{kn}} &\cdot& u_{k+1}^{\alpha_{k+1,n}} &\cdot& \dots &\cdot& u_{n}^{\beta_{n}} &=& u_{n}\\
    \end{matrix}
    \right.
  \end{equation}
Choose $1 \ne u_{k} \in \Omega_{1}(H_{k})$ as a nontrivial fixed
point of $\beta_{k}$, so that the first equation is
satisfied. Now the second equation reads
\begin{equation}\label{eq:next}
  u_{k}^{\alpha_{k,k+1}} =  u_{k+1}^{1 -\beta_{k+1}}.
\end{equation}
Since $\beta_{k+1}$ is \fpf\ on $\Omega_{1}(H_{k+1})$, the function $1
-\beta_{k+1}$    is     bijective    on    $\Omega_{1}(H_{k+1})$    by
Proposition~\ref{prpo:fpfbi},  so   that  there  is   a  $u_{k+1}  \in
\Omega_{1}(H_{k+1})$ which fulfills~\eqref{eq:next}. Proceeding in the
same fashion, we find values  for $u_{k+2}, \dots, u_{n}$ that satisfy
all the equations of~\eqref{eq:components2}, so that the resulting $u$
is a nontrivial fixed point of $\phi$ of order $p$.
\end{proof}

\section{Nilpotent endomorphisms}
\label{sec:nilp}

We record a couple of immediate facts here.
\begin{prop}\label{prop:nilp}
  Let $G$ be a finite group, $\phi$ a nilpotent endomorphism of $G$.
  
  Then $\phi$ is \fpf. Moreover, the following are equivalent
  \begin{enumerate}
  \item   $\phi$ is quasi-invertible, and
  \item   $\phi$ is abelian. 
  \end{enumerate}
\end{prop}

\begin{proof}
  If $\phi^{n} = 0$, and $x =  x^{\phi}$ for some $x \in G$, then $x =
  x^{\phi^{n}} = 1$.
  
  If $\phi$ is quasi-invertible, then $G^{\phi}$  is  abelian by
  Proposition~\ref{prop:abimage}. 

  %% The converse would follow from
  %% Theorem~\ref{thm:recab}, but we give the simple explicit form
  %% here.

  Conversely, let $G^{\phi}$ be  abelian. The inverse (with respect to
  map  composition) of  the map  $1 -  \phi$ is  clearly $1  +  \phi +
  \phi^{2} +  \dots +  \phi^{n-1}$, so if  $\phi$ has  a quasi-inverse
  $\psi$,  then  this  is  $\psi  =  - (\phi  +  \phi^{2}  +  \dots  +
  \phi^{n-1})$.   But $\psi$  is the  composition of  $\phi$  with the
  restriction   of  $-   (1  +   \phi  +   \dots  +   \phi^{n-2})$  to
  $G^{\phi}$. Since the latter group is  abelian, $- (1 + \phi + \dots
  +    \phi^{n-2})$   is    an   endomorphism    of    $G^{\phi}$   by
  Proposition~\ref{prop:commute}, so that $\psi$ is an endomorphism of
  $G$, and thus $\psi$ is indeed the quasi-inverse of $\phi$.
\end{proof}

\section{Piecing endomorphisms together}
\label{sec:piecingtogether}

Let $G$ be the semidirect product of the normal subgroup $K$ by the
subgroup $H$. Let $\theta \in \End(H)$,  $\eta \in \End(K)$.
We define a map $\phi$ on $G$, the semidirect product
of $K$ by $H$, by letting
\begin{equation}\label{eq:piece}
  (h k)^{\phi} = h^{\theta} k^{\eta},
\end{equation}
for $h \in H$ and $k \in K$. This will be an endomorphism of $G$ is
and only if for all $h, h' \in H$ and $k, k' \in K$
\begin{equation*}
  (h' k h k')^{\phi} 
  = 
  (h' h k^{h} k')^{\phi} 
  = 
  h'^{\theta} h^{\theta} (k^{h})^{\eta} k'^{\eta}
\end{equation*}
equals
\begin{equation*}
  (h' k)^{\phi}  (h k')^{\phi}
  =
  h'^{\theta} k^{\eta} h^{\theta} k'^{\eta}.
\end{equation*}
This means
\begin{equation*}
  (k^{h})^{\eta} = (k^{\eta})^{h^{\theta}},
\end{equation*}
for all $h \in H$ and $k \in K$,
that is, $\eta : K \to K$ is a twisted morphism of $H$-modules, or
equivalently, multiplying both sides of the previous equations on 
the left by $k^{-\eta}$,
\begin{equation}\label{eq:piecingtogether}
  [k, h]^{\eta} =   [k^{\eta}, h^{\theta}]
\end{equation}
for all $h \in H$ and $k \in K$.
Now we note
\begin{lemma}\label{lemma:vanish}
If $\phi$ is an abelian \fpf\ endomorphism of the finite group $G$,
then both terms 
of~\eqref{eq:piecingtogether} vanish for all $h \in H$ and $k \in K$.
\end{lemma}
\begin{proof}
  The right-hand term vanishes because $G^{\phi}$ is
  abelian and contains $H^{\theta}$ and $K^{\eta}$, and the left one vanishes
  since $[K, H] \le G' \cap K \le \ker(\phi) \cap K \le 
  \ker(\eta)$, again because $G^{\phi}$ is abelian.
\end{proof}

%% Now assume $H$ is abelian.
%% If $\phi$ defined in~\eqref{eq:piece} has to be a \fpf\ endomorphism
%% with a reciprocal, then by 
%% Proposition~\ref{prop:abimage} it has to satisfy $G' \le
%% \ker(\phi)$. Since $G' = K' [K, H] \le K$, and $\phi$ restricts to $\eta$ on
%% $K$, \eqref{eq:firstcond} holds. Since $G^{\phi}$ is
%% abelian, \eqref{eq:secondcond} also holds.

We may state
\begin{theorem}\label{thm:glue}
  Let the finite group $G$ be  the semidirect product of the group $K$
  by the  abelian group $H$. Let  $\theta$ be a  \fpf\ endomorphism of
  $H$, and  $\eta$ be  a nilpotent endomorphism  of $K$. Define  a map
  $\phi$ on $G$ by
  \begin{equation*}
    (h k)^{\phi} = h^{\theta} k^{\eta},
  \end{equation*}
  for $h \in H$ and $k \in K$. 
  
  Then the following are equivalent
  \begin{enumerate}
  \item $\phi$ is a quasi-invertible \fpf\ endomorphism, and
  \item $G' = K' [K, H] \le \ker(\eta)$, and $[K^{\eta}, H^{\theta}] = 1$.
  \end{enumerate}
\end{theorem}
Note  that  the  condition  $G'  \le  \ker(\eta)$  contains  both  the
condition $[K, H]  \le \ker(\eta)$, that we have  seen to be necessary
for $\phi$  to be an abelian  endomorphism, and the  condition $K' \le
\ker(\eta)$ which states that $\eta$ is an abelian endomorphism.

\begin{proof}
  We are only left with proving that $\phi$ has a quasi-inverse under the
  hypotheses on $\eta$, $\theta$, and the conditions of~(2). Now 
  by the results of Sections~\ref{sec:ab}~and \ref{sec:nilp}, the
  endomorphisms $\eta$ and $\theta$ have quasi-inverses
  $\eta'$ and 
  $\theta'$. By the conditions of (2), and
  Proposition~\ref{prop:abimage}, we have $[K, H] \le \ker(\eta) 
  = \ker(\eta')$ and $[K^{\eta'}, 
    H^{\theta'}] = [K^{\eta}, H^{\theta}] = 1$, so that the
  conditions
  for $\eta'$ and 
  $\theta'$ to induce an endomorphism $\psi$ of $G$ are
  satisfied. This will be the quasi-inverse of $\phi$, as for $h \in H$
  and $k \in K$ we have 
  \begin{equation*}
    (h k)^{-\phi + \phi \psi - \psi}
    =
    (h k)^{-\phi} (h k)^{\phi \psi}  (h k)^{-\psi}
    =
    k^{-\eta} h^{-\theta}  
    h^{\theta \theta'} k^{\eta \eta'}
    k^{-\eta'} h^{-\theta'}.
  \end{equation*}
  Now all terms commute, as $K^{\eta} = K^{\eta'}$, $H^{\theta} =
  H^{\theta'}$ are abelian and $[K^{\eta}, H^{\theta}] = 1$. So we
  have
  \begin{equation*}
    (h k)^{-\phi + \phi \psi - \psi}
    =
    h^{-\theta} h^{\theta \theta'}  h^{-\theta'} 
    k^{-\eta}  k^{\eta \eta'} k^{-\eta'}
    =
    h^{-\theta + \theta \theta' -\theta'}
    k^{-\eta  + \eta \eta' -\eta'}
    =
    1,
  \end{equation*}
  that is, $\psi$ is the quasi-inverse of $\phi$.
\end{proof}

We are now in a position to give the
\begin{proof}[Proof of Theorem~\ref{thm:recab}]
  Let $\phi$ be an abelian \fpf\ endomorphism of the group $G$. Decompose
  $G$ as in Theorem~\ref{thm:Fitting}, and let $\eta$ be the restriction
  of $\phi$ to $K$, and $\theta$ the restriction of $\phi$ to
  $H$. By Lemma~\ref{lemma:vanish}, and the fact that $\phi$ is
  abelian, conditions~(2) of 
  Theorem~\ref{thm:glue} are 
  satisfied, so that $\phi$ has a quasi-inverse.
\end{proof}

\section{Examples}
\label{sec:ex}

We review first the examples of~\cite[Section~5]{Childs} from the
point of view of this paper.

One of the examples is a Frobenius group (see \cite[4.5]{Gore},
\cite[Chap.~7]{Isaacs}) $G$ with 
Frobenius kernel $A$ and 
cyclic complement. (The kernel is taken to be abelian in the original example,
but this is immaterial.) 

So we have $A = G'$, and $(\Order{A}, \Index{G}{A}) = 1$. Let $\phi$
be a nontrivial \fpf\ endomorphism of $G$. We have $A = G' \le
\ker(\phi)$ by Proposition~\ref{prop:abimage}. Choose a Frobenius
complement $\Span{b}$ such that 
$G^{\phi} \le \Span{b}$, so that $\Span{b}$ is $\phi$-invariant. We
are left with determining the \fpf\ endomorphisms of $\Span{b}$. These
are, as noted in~\cite{Childs}, the maps $b \mapsto b^{s}$, with $(s -
1, \Size{b}) = 1$.  Proposition~\ref{prop:abfpf} yields the
equivalent condition $s \not\equiv 1 \pmod{p}$, for all primes
dividing $ \Size{b}$.

Clearly   a  group   of  order  $2$   has   no  nontrivial
\fpf\     automorphisms.      Thus      the     subgroup     $K$     of
Section~\ref{sec:Fitting} cannot have index $2$ in the group $G$.

This situation occurs in~\cite{Childs} first when $G = S_{n}$, the
symmetric group on $n \ge 3$ letters. Here $S_{n}' = A_{n} \le K$, so
we conclude as in~\cite{Childs} that $K = S_{n}$, and a nontrivial
abelian \fpf\ endomorphism  must be nilpotent, and thus map $G$ onto a
subgroup generated by 
an even involution. 

The case of dihedral groups is also described
in~\cite[Section~5]{Childs}; we review it here according to our
approach. 

Let $G$ be a dihedral group, $\phi$ be a nontrivial
abelian \fpf\ endomorphism  of $G$, and $H$ and $K$ be the
subgroups of Section~\ref{sec:Fitting}. Fix an element $x$ such that
$\Span{x}$ has index $2$ in $G$,  and  an involution $y  \notin
\Span{x}$. 

If $G$ has order twice an odd number, then $G' = \Span{x}$ has index
$2$. Since $K \ge \ker(\phi) =  G'$, by the above argument 
$\phi$ should
nilpotent, with $G^{\phi}$ a subgroup of order $2$ of $\ker(\phi) =
\Span{x}$. But the latter is a group of odd order, so there are no
abelian \fpf\ endomorphisms in this case.

If $G$ has order $4 m$, then $G' = \Span{x^{2}}$ has order $m$. We
first discuss what $K \ge G'$ can be.

If $K = G'$, then $G'$ must  have a complement $H$. This occurs if and
only  if $m$ is  odd, and  then $H$  is one  of the  Klein four-groups
$\Span{x^{m}, x^{i}y}$. We obtain  $\phi$ by extending the identity on
$K$ by  a (\fpf) automorphism  of order $3$  of $H$. This  covers
case   (5)  with   $i$   odd,  and   case  (7)   with  $i$   even
of~\cite[Section~5]{Childs}.

If $K > G'$,  then $K$ cannot have index $2$, as  noted above, so that
$K =  G$, and $\phi$ is  nilpotent.  We discuss  the possibilities for
$\ker(\phi) \ge G'$.

If $\ker(\phi) = G'$, then  $G^{\phi}$ is one of the Klein four-groups
$\Span{x^{m},  x^{i}y}$.   These  have  to   intersect  $\ker(\phi)  =
\Span{x^{2}}$  in $\Span{x^{m}}$,  and  then $m$  must  be even:  this
covers case (5) with $i$ even, and case (7) with $i$ odd.

If $\ker(\phi) > G'$, there are three possibilities for
the maximal subgroup $\ker(\phi)$ of $G$. (Here $G^{\phi}$ is a
subgroup of order $2$ of $\ker(\phi)$.)
\begin{enumerate}
\item If $\ker(\phi) = \Span{x}$, then $G^{\phi} = \Span{x^{m}}$, and
  we get case (1). 
\item If $\ker(\phi)  = \Span{x^{2}, y}$, then either  $m$ is even and
  $G^{\phi} = \Span{x^{m}}$ (case  (3)), or $G^{\phi} = \Span{x^{2 a}
  y }$ for some $a$ (case (4)).
\item If $\ker(\phi) = \Span{x^{2}, x y}$, then either $m$ is even and
  $G^{\phi} = \Span{x^{m}}$ (case (2)),  or $G^{\phi} = \Span{x^{2 a +
    1} y}$ for some $a$ (case (6)).
\end{enumerate}

We now give a couple more examples.

Consider a semidirect product $G = H K$, with $K$ normal in $G$,
and $H$ abelian. Clearly 
all \fpf\ endomorphisms of $H$ induce abelian \fpf\ endomorphisms $\phi$
of $G$ with $K \le 
\ker(\phi)$. In all the examples so far, and in the notation of
Theorem~\ref{thm:Fitting}, we have had either $K = G$ (that is, $\phi$
is nilpotent), or $\phi$
trivial on $K$. We now give an example in which $K < G$, and $\phi$
acts nontrivially on $K$.

Let $K$ be a (nonabelian) special $p$-group \cite[III.13]{Hup}, that
is, a group in which $K' = Z(K) = \Frat(K)$ is elementary
abelian. (Here $\Frat(K) = K' K^{p}$ is the Frattini subgroup of $K$.)
The endomorphisms in $\Hom(K, Z(K)) \cong \Hom(K/K', Z(K))$ are then
all nilpotent. Consider the elementary abelian group $H$ of central
automorphisms of 
$K$,
\begin{equation*}
  H = \Set{1 + f : f \in \Hom(K, Z(K))},
\end{equation*}
and let $G$ be the natural extension of $K$ by $H$. Define $\eta$ on
$K$ to be any element of $\Hom(K, Z(K))$, and $\theta$ to be any
\fpf\ endomorphism of $H$. (We might for instance regard the
elementary abelian group $H$ as the additive group of a finite field $E$,
and take $\theta$ to be the multiplication by an element of $E$ different
from $0, 1$.)

Clearly $[K, H] = Z(K) \le \ker(\eta)$, and
$K^{\eta} \le Z(K)$ commutes with $H^{\theta} \le H$. The recipe of
Section~\ref{sec:ab} is thus satisfied, so that this
defines a quasi-invertible \fpf\ endomorphism of $G$.

However, Childs shows in~\cite[Theorem 2]{Childs} that two
\fpf\ endomorphisms $\phi, \psi$ of the finite group $G$ induce the same regular
subgroup of the group of permutations on $G$ (as in the Introduction)
if and only if there is a \fpf\ endomorphim $\zeta: G \to G$, with
$G^{\zeta} \le Z(G)$, such that $1 - \phi = (1 - \zeta)(1 - \psi)$, or
$\phi = \psi - \zeta \psi + \zeta$. We say that two abelian
endomorphisms \fpf\ $\phi, \psi$ are \emph{equivalent} if they satisfy
this condition.

Now it is not difficult to see that in the last example $\phi$ is
equivalent in this sense to another abelian \fpf\ endomorphism which
acts trivially on $K$.

To see an example where this does
not happen, take $K$ to be the 
nonabelian $p$-group of order $p^{3}$ and exponent $p$, for $p$
odd. If $K = \Span{a, b}$, %%and $z = [b, a]$, 
let $\eta$ be the
nilpotent endomorphism of 
$K$ defined by $a \mapsto b \mapsto 1$, so $K^{\eta} =
\Span{b}$ is abelian, but 
not central in $K$. We may then consider for instance the automorphism
$\alpha$ of
$K$ of order $p$ which acts as $a \mapsto a b$,  $b \mapsto b$, and let $H =
\Span{\alpha}$. For $\theta$ we may take any \fpf\ endomorphism of $H$,
that is, any map $\alpha \mapsto \alpha^{s}$, with $s \not\equiv 1
\pmod{p}$.
If $G = H K$, then we have $G' = \Span{b, z} \le \ker(\eta)$, and
$[K^{\eta}, H^{\theta}] 
\le [\Span{b}, H] = 1$, so that the recipe of
Section~\ref{sec:ab} is satisfied.

%% \input{Addendum}

%%\bibliography{Refs}
\providecommand{\noopsort}[1]{}
\providecommand{\bysame}{\leavevmode\hbox to3em{\hrulefill}\thinspace}
\providecommand{\MR}{\relax\ifhmode\unskip\space\fi MR }
% \MRhref is called by the amsart/book/proc definition of \MR.
\providecommand{\MRhref}[2]{%
  \href{http://www.ams.org/mathscinet-getitem?mr=#1}{#2}
}
\providecommand{\href}[2]{#2}

\end{document}